\documentclass[a4paper, 11pt]{amsart}

\usepackage{comment}
\usepackage[T1]{fontenc}
\usepackage[utf8]{inputenc}
\usepackage{geometry}
\usepackage{amssymb}
\usepackage{hyperref}

\makeatletter
\@namedef{subjclassname@2020}{%
  \textup{2020} Mathematics Subject Classification}
\makeatother

\allowdisplaybreaks[2]

\newtheorem{theorem}{Theorem}[section]

\newtheorem{lemma}[theorem]{Lemma}
\newtheorem{corollary}[theorem]{Corollary}
\theoremstyle{definition}

\newtheorem{example}[theorem]{Example}

\numberwithin{equation}{section}

\newcommand\her{\mathrm{her}}
\newcommand\Crm{\mathrm{C}}

\newcommand\RR{\mathbb{R}}
\newcommand\NN{\mathbb{N}}
\newcommand\CC{\mathbb{C}}
\newcommand\ZZ{\mathbb{Z}}

\begin{document}

\title[Products of commutators in simple algebras]{Products of commutators in simple algebras}

\author{Matej Brešar} 
\author{Hau-Yuan Jang}
\author{Leonel Robert}
\address{Faculty of Mathematics and Physics, University of Ljubljana \&
Faculty of Natural Sciences and Mathematics, University of Maribor \& IMFM, Ljubljana, Slovenia}
\email{matej.bresar@fmf.uni-lj.si}
\address{Department of Mathematics, National Cheng Kung University, Tainan, Taiwan}
\email{l18121022@gs.ncku.edu.tw}
\address{Department of Mathematics, University of Louisiana at Lafayette, Lafayette, LA 70503, USA}
\email{lrobert@louisiana.edu}

\thanks{The first author was partially supported by the ARIS Grants P1-0288 and J1-60025.
}



\begin{abstract}
    Let $A$ be a finite-dimensional simple algebra that is not a field. We show that every $a\in A$ can be written as 
    $a=(bc-cb)(de-ed)$ for some
    $b,c,d,e\in A$. This is not always true for infinite-dimensional simple algebras. In fact, for any $m\in \NN$ we provide an example of an infinite-dimensional simple unital $C^*$-algebra $A$ in which $1$ cannot be written as $\sum_{i=1}^m x_i(a_ib_i-b_ia_i)y_i$ for some $x_i,a_i,b_i,y_i\in A$. 
\end{abstract}

\maketitle
\section{Introduction}
In the recent paper \cite{GT}, Gardella and Thiel showed that a unital ring $A$ is equal to its ideal generated by all (additive) commutators $[x,y]=xy-yx$ if and only if there exists an $N\in \NN$  such that every $a \in A$
can be written as 
\[
a=\sum_{j=1}^N [b_j,c_j][d_j,e_j]
\] 
for some  $b_j , c_j , d_j , e_j\in A$.
The obvious problem that arises is to find the smallest such
$N$. It is easy to see that 
$N\le 2$ if $A$ is a noncommutative division ring \cite[Proposition 5.8]{GT}, but it is  far less obvious whether $N$ is actually equal to $1$ \cite[Question 5.9]{GT}. This was shown to be true in the special case where $A$ is a  skew Laurent series division ring  
 \cite{JK}.

 In Section \ref{s2}, we show that
$N=1$ if $A$ is a finite-dimensional simple algebra that is not a field (Theorem \ref{t1}). 
In particular, this gives a positive answer to the aforementioned question by Gardella and Thiel for finite-dimensional division algebras. In fact, the division algebra case is the main novelty  since     the case of   
finite-dimensional simple algebras different from division algebras was already handled \cite[Corollary 4.5]{BGT}. However, we will also give a different, short proof  of this latter result.

We remark that Theorem \ref{t1}  can also be viewed as a contribution to the general problem of whether the image of a multilinear polynomial on an algebra is a vector space (see \cite{KBMRY}).

In Section \ref{s3}, we complement the results of Section \ref{s2} by showing that $N$ is not uniformly bounded among all simple, unital, infinite-dimensional C*-algebras. In fact, a stronger failure occurs. It is clear that
a unital ring $A$ is equal to the ideal generated by its commutators if and only if there exists an $m\in \NN$  such \[
1=\sum_{i=1}^m x_i[a_i,b_i]y_i
\]   
for some $x_i,y_i,a_i,b_i\in A$ (and in this case $m\leq N$).
In Section \ref{s3} we show
that for each $m\in \NN$ there exists a simple, unital, infinite-dimensional C*-algebra $A$ in which  the above equation
cannot be satisfied with that $m$ (Theorem  \ref{thm:simplenoCm}).
To prove this  we adapt the proof of \cite[Theorem 1.4]{robert2015}, which in turn relies on  Villadsen's technique  from \cite{villadsen}.  As a corollary, we deduce that for any noncommutative polynomial $f$ (with coefficients in $\CC$) that is an identity for $\CC$, there exists a simple, unital, noncommutative C*-algebra  whose image under $f$ does not contain the unity (Corollary \ref{cor:fimage}).

\section{The finite-dimensional case}\label{s2}

We start with an elementary lemma,
which will be needed for obtaining a new proof of \cite[Corollary 4.5]{BGT}. This lemma will actually reduce the problem to the situation where we can apply  
the following deeper result by Amitsur and Rowen:  If $A$ is a (finite-dimensional) central simple algebra that is not a division algebra, then every noncentral element in $A$ whose reduced 
trace is $0$ is a commutator \cite[Theorem 2.4]{AR}. (To the best of our knowledge, it is still unknown whether this is true for 
arbitrary elements of reduced trace 0 in arbitrary central simple algebras.)

\begin{lemma}\label{thel}
Let $F$ be a field, 
    let $D$ be a central division $F$-algebra of degree $n > 1$,
    let $m\in \NN$, and let $A= M_m(D)$.
If $H$ is a hyperplane of $A$,
then every element $a\in A$ can be written as $a=h_1h_2$ for some
    $h_1,h_2\in H$.  
    Moreover, identifying $A$ with 
$M_m(F)\otimes_F D$, we have that
given any invertible matrix
    $t\in M_m(F)$ we can choose
    $h_1,h_2$ to be of the form
    $h_1=a(t^{-1}\otimes d^{-1})$
    and $h_2= t\otimes d$ for some 
    $d\in D\setminus{\{0\}}$.
\end{lemma}

\begin{proof}  Let 
$K$ be any maximal subfield of $D$, let
$\tau$ be any linear functional on $A$ such that
$H=\ker\tau$, and let  $t\in M_m(F)$ be
 any invertible matrix. 

 We claim that there
exists a $d_0\in D\setminus\{0\}$ such that $t\otimes kd_0\in H$
for every $k\in K$. To prove this, we use the idea from \cite{H}   and define an $F$-linear map
$U$ from $D$ to the dual $K^{*}$  of $K$
by $U(d)(k)=\tau(t\otimes kd)$ for all $d\in D$, $k\in K$. 
As $\dim_F D=n^2$ and 
$\dim_F K^*=n$, 
it follows that 
$U$ has  nontrivial kernel. Any  $d_0\in \ker U\setminus\{0\}$ satisfies our claim.

We may assume that $a\ne 0$. Then
 $a(t^{-1}\otimes d_0^{-1}K)$
is  an $n$-dimensional space over $F$, so its intersection with $H$ is nontrivial.
Choose  $k_0\in K\setminus\{0\}$ such that 
$h_1=a(t^{-1}\otimes d_0^{-1}k_0)\in H$.  The element $h_2=t\otimes k_0^{-1}d_0$ 
lies in $H$  by what we proved in the preceding paragraph. Clearly,
 $a=h_1h_2$ and $h_1,h_2$ are of the desired form.
\end{proof}

The case where $n=1$, i.e., the case where $A=M_m(F)$, is more difficult. It turns out that the conclusion that   $a=h_1h_2$ for some
    $h_1,h_2\in H$ holds for $m\ge 3$, but not always for $m=2$; see \cite[Proposition 12 and Theorem 14]{de}.

We  now state and prove our main result of this section.

\begin{theorem}\label{t1}
Every element in a finite-dimensional simple algebra $A$
that is not a field is  a product of two commutators.
\end{theorem}

\begin{proof}  The center 
of $A$ is a field that contains scalar multiples of unity. We may therefore consider $A$ as a
central simple algebra, i.e., as a (finite-dimensional) 
algebra over its center $F$. 
By  Wedderburn's theorem, 
$A\cong M_m(D)$ for some $m\ge 1$ and a central  division $F$-algebra $D$.
We divide the proof into two cases, when $m=1$ and when $m> 1$.

\smallskip 
\underline{First case: $m=1$.} We are assuming that $A=D$ is a central division algebra.
Take $d\in D$ and let us show that it is a product of two commutators.
We may assume that $d\ne 0$.
 Let $L$ be a maximal subfield of $D$ containing $d$.

Assume first that there exists an $\ell\in L$ that is not separable over $F$. Then we may  apply \cite[Theorem 0.7]{AR}, which tells us that $1 =[\ell,a] $ for some $a\in D$. Since $d$ and $\ell$ lie in $L$ and therefore commute,  it follows that $d=d[\ell,a][\ell,a]= [\ell,da][\ell,a]$.

We may thus assume that   $L$ is a separable extension of $F$. By the primitive element theorem,  $L=F(u)$
for some $u\in L$. 
Fix $v\in D$ such that $[u,v]\ne 0$ and set
 $$W=[u,v]^{-1}(L + Fv).$$
  We claim that the inverse of every nonzero element  $w\in W$
 is a commutator. Write $$w=[u,v]^{-1}(\ell + \lambda v),$$ where $\ell\in L$ and $\lambda \in F$.
If $\lambda=0$
then $$w^{-1}= \ell^{-1}[u,v]= [u,\ell^{-1}v],$$ and if  $\lambda\ne 0$ then $$w^{-1}= (\ell + \lambda v)^{-1}[u,v] =\lambda^{-1}(\ell + \lambda v)^{-1}[u,\ell + \lambda v]= [\lambda^{-1}(\ell + \lambda v)^{-1}u,\ell + \lambda v]. $$
This proves our claim.

As usual, we denote by ad$_u$ the inner derivation given by
ad$_u(y) =[u,y]$, $y\in D$. Since 
$L=F(u)$ is a maximal subfield, it is equal to the kernel of ad$_u$. Write $n$ for the degree of $D$. As $L=\ker {\rm ad}_u$ has dimension $n$, it follows that ad$_u(D)$, the image of ad$_u$, has dimension $n^2-n$.
Note that the space $W$ has dimension $n+1$. The same is then true for the space $dW$. Consequently, ad$_u(D)\cap dW\ne\{0\}$. Therefore, there exist a
  $b\in D$
 and a $w\in W\setminus{\{0\}}$ such that $[u,b]= dw$, i.e., $d=[u,b]w^{-1}$.
 Since $w^{-1}$ is a commutator, this completes the proof for this case.

\smallskip 
\underline{Second case: $m>1$.} 
 The case where $D=F$ is of different nature and was treated many years ago \cite[Theorem 4.1]{Botha} (and is also covered by  results of \cite{de}). We therefore assume that the degree $n$  of $D$ is greater than $1$. 

Take a nonzero $a\in A$. By \cite[Theorem 2.4]{AR}, it is enough to show that there exist
 elements $h_1,h_2\in A$ that are noncentral, have reduced trace $0$, and satisfy   $a=h_1h_2$. Since the set of  elements having reduced trace  $0$ is a hyperplane, this follows from 
 Lemma \ref{thel}. Indeed, we just choose 
  an invertible, nonscalar matrix $t\in M_m(F)$ such that 
$a\notin t\otimes D$ to guarantee that $h_1=a(t^{-1}\otimes d^{-1})$
    and $h_2= t\otimes d$  are noncentral.
\end{proof}

\section{An infinite-dimensional counterexample}\label{s3}
Recall that a C*-algebra is a Banach *-algebra satisfying the C*-identity $\|x^*x\|=\|x\|^2$ for all elements $x$.

For a unital C*-algebra $A$ and $m\in \NN$, consider the following property:
\begin{itemize}
	\item[$\Crm_m$]: $1=\sum_{i=1}^m x_i[a_i,b_i]y_i$ for some $x_i,y_i,a_i,b_i\in A$.
\end{itemize}	
Every simple, unital, noncommutative  C*-algebra has $\Crm_m$ for some $m$. In Theorem \ref{thm:simplenoCm}
we show that there is no uniform bound on  $m$ across all such C*-algebras. To outline the argument, fix $m$. First, we construct homogeneous C*-algebras  that fail $\Crm_m$ (Theorem \ref{thm:noCm}). These C*-algebras have the form $pM_n(C(X))p$ for a suitable topological space $X$ and projection $p\in M_n(C(X))$.
Next, we construct a sequence of such homogeneous C*-algebras $A_1,A_2,\ldots$, together with embeddings $\phi_k\colon A_k\to A_{k+1}$, such that the inductive limit C*-algebra $A=\varinjlim A_k$ is simple and still fails $\Crm_{m}$ (Theorem \ref{thm:simplenoCm}).


Let $X$ be a compact Hausdorff space. Let $C(X)$ denote the C*-algebra of continuous $\CC$-valued functions on $X$. It will be helpful to bear in mind the correspondences among projections in $\bigcup_{n=1}^\infty M_n(C(X))$, finitely generated projective $C(X)$-modules, and complex vector bundles over $X$. More concretely, given a projection $p\in M_n(C(X))$, we obtain a vector bundle $\xi_p=(X,E,\pi)$, where 
\[
E=\{(x,v)\in X\times \CC^n:p(x)v=v\},
\] 
and $\pi\colon E\to X$ is the projection onto the first coordinate. The set
\[
P_p=\{s\in C(X,\CC^n) : p(x)s(x)=s(x)\hbox{ for all }x\in X\}
\] 
of continuous sections
of $\xi_p$ is a finitely generated projective module over  $C(X)$ (with pointwise defined module operations).

The Murray-von Neumann subequivalence  of projections $p,q\in \bigcup_{n=1}^\infty M_n(C(X))$ is defined as follows: $p\precsim q$ if $p=v^*v$ and $vv^*\leq q$ for some $v\in \bigcup_{n=1}^\infty M_n(C(X))$. We have $p\precsim q$ if and only if $P_p$ embeds in $P_q$ as a $C(X)$-submodule.

Given $p\in M_m(C(X))$ and $q\in M_n(C(X))$, let us denote by  $p\oplus q\in M_{m+n}(C(X))$ the projection 
\[
\begin{pmatrix}
	p&0\\
	0&q	
\end{pmatrix}.	
\]
Notice that $\xi_{p\oplus q}\cong \xi_{p}\oplus \xi_q$ and $P_{p\oplus q}\cong P_p\oplus P_q$.

Denote by $e_{11}\in M_\infty(C(X))$ the projection whose $(1,1)$-entry is equal to 1 and all other entries are zero.
\begin{lemma}
	Let $p\in M_n(C(X))$ be a projection. Then $e_{11} \precsim  p$ if and only if  $\xi_p$ admits a nowhere-zero continuous section, i.e., there exists 	$s\in P_p$ with $s(x)\neq 0$ for all $x$.	
\end{lemma}	

\begin{proof}
	If $e_{11}=s^*s$ and $ss^*\leq p$, then 
	$s\in pM_n(C(X))e_{11}$. We can think of $s$ as an element of $P_p$, i.e., a continuous section of $\xi_p$. From $s^*s=e_{11}$
	we get that $s(x)\neq 0$ for all $x$. 
	
	Suppose conversely that there exists a nowhere-zero continuous section of $\xi_p$. We thus get  $s\in pM_n(C(X))e_{11}$
	that is nonzero for all $x$. Then $s^*(x)s(x)=\lambda(x)e_{11}$, where $\lambda(x)>0$ for all $x$. So 
	$\tilde s(x)=\frac{1}{\lambda(x)^{\frac12}}s(x)$ implements the Murray-von Neumann comparison between $e_{11}$ and $p$ i.e., $\tilde s^*\tilde s=e_{11}$ and $\tilde s\tilde s^*\leq p$. 	
\end{proof}	

We will use the following consequence of the above lemma: $e_{11}\precsim p^{\oplus n}$ for some $n\in \NN$
if and only if  there exist  $s_1,s_2,\ldots,s_n\in P_p$ such that  for all
$x\in X$ at least one $s_j(x)\neq 0$. This is because elements of  $P_p^{\oplus n}$ are  $n$-tuples of elements of $P_p$.
Notice that this means that if $e_{11}\not\precsim p^{\oplus n}$, then for any $n$-tuple
of elements in $P_p$ (sections of $\xi_p$) there exists $x\in X$ at which they all vanish.

Using standard  tools from algebraic topology we can construct projections whose associated vector bundles 
have no nowhere-zero sections. Before that, let us see how this is useful for our goal.

Given a projection $p\in M_\infty(C(X))$, let $\her(p)$ denote the algebra $pM_\infty(C(X))p$.

\begin{theorem}\label{thm:noCm}
	Let $m\in \NN$. Let $p\in M_N(C(X))$ be a projection such that $e_{11}\not\precsim p^{\oplus 8m}$. Define  $q=e_{11}\oplus p\in M_{N+1}(C(X))$
	and  $A=\her(q)\subseteq M_{N+1}(C(X))$. Then $A$ fails  $\Crm_m$.	
\end{theorem}	

\begin{proof}
	Let us represent elements of  $A$  as $2\times 2$ matrices
	\[
	\begin{pmatrix}
		a & b\\
		c & d
	\end{pmatrix},
	\]
	where $a\in C(X)$,  $b\in e_{1}^tM_{N+1}(C(X))p$, $c\in pM_{N+1}(C(X))e_{1}$, and $d\in pM_N(C(X))p$.
	Here $e_1$ is the column vector with 1 in the first entry and zeros elsewhere. 
	Observe that both $b^t$ and $c$ are elements of $P_p$.
	
	Now suppose that we have 
	\begin{equation}\label{Cmeqn}
		1=\sum_{i=1}^m x_i[a_i,b_i]y_i
	\end{equation}
	for some $x_i,y_i,a_i,b_i\in A$. Write
	\begin{align*}
		x_i=	\begin{pmatrix}
			x_{11}^{(i)} & x_{12}^{(i)}\\
			x_{21}^{(i)} & x_{22}^{(i)}
		\end{pmatrix},\,
		a_i=	\begin{pmatrix}
			a_{11}^{(i)} & a_{12}^{(i)}\\
			a_{21}^{(i)} & a_{22}^{(i)}
		\end{pmatrix},\\
		b_i=	\begin{pmatrix}
			b_{11}^{(i)} & b_{12}^{(i)}\\
			b_{21}^{(i)} & b_{22}^{(i)}
		\end{pmatrix},\,
		y_i=	\begin{pmatrix}
			y_{11}^{(i)} & y_{12}^{(i)}\\
			y_{21}^{(i)} & y_{22}^{(i)}
		\end{pmatrix}
	\end{align*}
	for $i=1,\ldots,m$. The off-diagonal terms of each of these matrices 
	are sections of $\xi_p$, i.e., an element of $P_p$. We have $8m$ such elements.
	By the assumption   $e_{11}\not\precsim p^{\oplus 8m}$, there exists $x\in X$ on which
	all the off-diagonal entries simultaneously vanish. Evaluating on such an $x$ both sides of \eqref{Cmeqn} 
	and comparing the top-left corners we get  $1=0$ (since $C(X)$ is commutative), a contradiction. Thus $A$ fails to have the property $\Crm_m$.
\end{proof}

Next we examine some examples of projections such that $e_{11}\not\precsim p$. To justify that $e_{11}\not\precsim p$ in these examples we rely on the fact that if $e_{11}\precsim p$, then $\xi_{p}$ has a nowhere-zero section, and this in turn implies that the Euler class of $\xi_p$ is zero (see \cite[Property 9.7]{milnor-stasheff}, \cite[Proposition 3.13 (e)]{VBKT}). Thus, if the Euler class of $\xi_p$ is nonzero, then  $e_{11}\not\precsim p$.

Let us denote by $e(\xi_p)\in H^*(X)$ the Euler class of the (complex) vector bundle $\xi_p$, where $H^*(X)$ is the 
cohomology ring of $X$ with integer coefficients. If $p$ has pointwise constant rank $k$, then $\xi_p$   has dimension $k$ over $\CC$, and $2k$ over $\RR$, so  $e(\xi_p)\in H^{2k}(X)$. We note that if $p$ has rank $k$, then
$e(\xi_p)$ agrees with the $k$-th Chern class $c_k(\xi_p)$.

\begin{example}[Bott projection]\label{bottprojection}
	The set $\mathcal P_2$	of rank one projections in $M_2(\CC)$ is homeomorphic to $S^2$, the 2-dimensional sphere, via the map
	$p\colon S^2\to \mathcal P_2\subseteq M_2(\CC)$:
	\[
	p((x,y,z))=\frac12
	\begin{pmatrix}
		1+x & y-iz\\
		y+iz & 1-x	
	\end{pmatrix}.
	\]  
	Regard $p$ as a projection  in $M_2(C(S^2))$. Then $e(\xi_p)$ is a generator of $H^2(S^2)$. In fact, we have a ring isomorphism
	$H^*(S^2)\cong \ZZ[\alpha]/(\alpha^2=0)$ such that  $e(\xi_p)\mapsto \alpha$ \cite[Theorem 14.4]{milnor-stasheff}. It follows that $\xi_p$ does not admit nowhere-zero  sections. That is,   $e_{11}\not\precsim p$. 
\end{example}	

\begin{example}\label{botttensor}
	Let $n\in \NN$ and set $X=S^2\times S^2\times\cdots \times S^2$ ($n$-fold cartesian product). Let $p\in M_2(C(S^2))$ be 
	the Bott projection (from the previous example).  Consider $q\in M_{2^n}(C(X))$
	defined as
	\[
	q(x_1,x_2,\ldots, x_n)=p(x_1)\otimes p(x_2)\otimes \cdots \otimes p(x_n)\in M_{2^n}(\CC)
	\]
	for $(x_1,x_2,\ldots,x_n)\in (S^2)^n$. Let us argue that $e_{11}\not\precsim q^{\oplus n}$. 
	From the  ring isomorphism   $H^*(S^2)\cong \ZZ[\alpha]/(\alpha^2=0)$, where $\alpha\in H^2(S^2)$, and the K\"{u}nneth
	formula, we have that 
	\[
	H^*(X)\cong \ZZ[\alpha_1,\alpha_2,\ldots,\alpha_n]\,/\,(\alpha_j^2:j=1,\ldots,n),
	\] 
	where $\alpha_1,\ldots,\alpha_n\in H^2(X)$ are induced by the projections $\pi_j\colon X\to S^2$ onto each sphere factor
	and the choice of a generator in $H^2(S^2)\cong \ZZ$.  Let us identify $H^*(X)$ with the ring on the right-hand side.
	Notice now that the vector bundle  $\xi_q$
	is the external tensor product of $\xi_p$ over each factor $S^2$ in $X$. Put differently, $\xi_q$ is the internal tensor
	product of the pullbacks $\xi_{p_i}$ of $\xi_p$ along each projection map $\pi_i\colon X\to S^2$. We thus get
	\[
	e(\xi_q)=\sum_{i=1}^n e(\xi_{p_i})=\sum_{i=1}^n \alpha_i\in H^2(X)
	\]
	\cite[Proposition 3.10]{VBKT}.
	Then, taking direct sum of $\xi_q$ with itself we get 
	\[
	e(\xi_{q^{\oplus n}})=(e(\xi_q))^n=(\alpha_1+\cdots +\alpha_n)^n=n!\alpha_1\alpha_2\cdots\alpha_n\in H^{2n}(X).
	\]
	(see \cite[Formula (4.7)]{milnor-stasheff}, \cite[Proposition 3.13 (b)]{VBKT}).
	Since $e(\xi_{q^{\oplus n}})\neq 0$, the vector bundle $\xi_{q^{\oplus n}}$ does not admit a nowhere-zero section.  
	That is, any $n$-tuple of elements  of $P_q$ must simultaneously vanish at some point $x\in (S^2)^n$.
\end{example}	

We will need one more kind of example.

\begin{example}\label{ex:directsum}
	Let  $l_i\leq n_i$ for $i=1,\ldots,k$ be in $\NN$.
	Let $X_i=(S^2)^{n_i}$ for $i=1,\ldots,k$ and $Y=X_1\times X_2\times\dots \times X_k$.
	Let $p\in M_2(C(S^2))$ be as in Example \ref{bottprojection} and  $q_i\in M_{2^{n_i}}(C(X_i))$
	as in  Example \ref{botttensor}. Define a projection 
	$r\in M_{\infty}(C(Y))$ as
	\[
	r(x_1,\ldots,x_k)=q_1^{\oplus l_1}(x_1)\oplus q_2^{\oplus l_2}(x_2)\oplus \cdots \oplus q_k^{\oplus l_k}(x_k) 
	\]
	for $(x_1,\ldots,x_k)\in X_1\times \cdots \times X_k$.  
	
	Suppose now that  $nl_i\leq n_i$ for all $i$, and let us argue that 
	$e_{11}\not\precsim r^{\oplus n}$. The Euler class calculations are very much as in the previous example:
	\[
	e(\xi_{r^{\oplus n}})=(e(\xi_r))^n=\Big(\sum_{j=1}^{n_1} \alpha_{1,j}\Big)^{nl_1}\Big(\sum_{j=1}^{n_1} \alpha_{2,j}\Big)^{nl_2}\cdots \Big(\sum_{j=1}^{n_k} \alpha_{k,j}\Big)^{nl_k}.
	\]
	The product on the right-hand side is calculated in the ring 
	\[
	\ZZ[\alpha_{i,j}:i=1,\ldots,k,\, j=1,\ldots,n_i]/(\alpha_{i,j}^2=0).
	\]
	The inequality $nl_i\leq n_i$ guarantees that this product is nonzero.
	Thus, $\xi_{r^{\oplus n}}$ does not admit nowhere-zero sections. 
	That is, any $n$-tuple of elements of $P_r$ must simultaneously vanish at some
	point $y\in Y$.
\end{example}

Combining Theorem \ref{thm:noCm} with the previous examples  we obtain C*-algebras of the form $\her(e_{11}\oplus r)$, with
$r\in M_\infty(C(X))$, that fail  $\Crm_m$. 
Next we construct the embeddings between such C*-algebras  that we will use to build a simple inductive limit.

Let $p\in M_n(C(X))$ be a projection and $x_0\in X$. We get a *-homomorphism
$\mathrm{ev}_{x_0}$ from $\her(p)$ to $\her(p(x_0))$ by evaluating at $x_0$:
\[
\her(p)\ni f\mapsto f(x_0)\in \her(p(x_0))\subseteq M_{n}(\CC). 
\]
Let $l=\mathrm{rank}(p(x_0))$. Then $\her(p(x_0))$ is isomorphic to $M_l(\CC)$. (This can be seen, for example,  by diagonalizing $p(x_0)$.) Choosing any isomorphism between $\her(p(x_0))$ and $M_l(\CC)$, let us regard $\mathrm{ev}_{x_0}$ as a *-homomorphism from $\her(p)$ to $M_l(\CC)$.

Let $q\in M_\infty(C(Y))$ and $l\in \NN$. Then we can embed $M_l(\CC)$ in $\her(q^{\oplus l})$ via 
\[
(\lambda_{ij})_{ij}\mapsto 
\lambda\otimes q:=(\lambda_{ij}q)_{ij}.
\]

We can combine the two previous constructions as follows. Let $p\in M_\infty(C(X))$ and  $q\in M_\infty(C(Y))$ be projections. Let $x_0\in X$. Set $l=\mathrm{rank}(p(x_0))$. Then we get a *-homomorphism $\her(p)\to \her(q^{\oplus l})$ given by 
\[
\her(p)\ni f\mapsto \mathrm{ev}_{x_0}(f)\otimes q\in\her(q^{\oplus l}).
\]
This homomorphism is not in general an embedding (unless $X$ is a singleton). Define now a projection $r\in M_\infty(C(X\times Y))$
as $r=(p\circ \pi_X)\oplus (q\circ\pi_Y)^{\oplus l}$, i.e., 
\[
r(x,y)=p(x)\oplus q(y)^{\oplus l}\hbox{ for }(x,y)\in X\times Y.
\]
Let
 $\phi\colon \her(p)\to \her(r)$ be the *-homomorphism defined as
\[
\phi(f) = 
\begin{pmatrix}
	f\circ\pi_X & 0\\
	0 & \mathrm{ev}_{x_0}(f)\otimes (q\circ\pi_Y)	 
\end{pmatrix}\in \her(r).
\]
Then $\phi$ is a unital embedding from $\her(p)$ to $\her(r)$.

\begin{lemma}\label{lem:inductive}
	Let $(A_n)_{n=1}^\infty$ be an increasing sequence of unital C*-subalgebras of $A$ such that $A=\overline{\bigcup_n A_n}$.
	If $A_n$ fails $\Crm_m$ for all $n$, then $A$ fails $\Crm_m$.	
\end{lemma}	
\begin{proof}
	Suppose for the sake of contradiction that 	
	\[
	1=\sum_{i=1}^m x_i[a_i,b_i]y_i
	\]
	for some $x_i,y_i,a_i,b_i\in A$. Using that $\bigcup_n A_n$ is dense
	in $A$ and that the sequence $(A_n)_n$ is increasing, we find $n$ and  $x_i',y_i',a_i',b_i'\in A_n$ such that
	\[
	\|1-\sum_{i=1}^m x_i'[a_i',b_i']y_i'\|<1.
	\]
	It follows that $z=\sum_{i=1}^m x_i'[a_i',b_i']y_i'$ is invertible in $A_n$, and so
	\[
	1=\sum_{i=1}^m z^{-1}x_i'[a_i',b_i']y_i'.
	\]
	This contradicts that $A_n$ fails to have $\Crm_m$.
\end{proof}

\begin{theorem}\label{thm:simplenoCm}
	Let $m\in \NN$. There exists a simple unital C*-algebra that fails to have property $\Crm_m$.	
\end{theorem}

\begin{proof}
	Let $(k_n)_{n=1}^\infty$ be an increasing sequence of natural numbers. 
	We will use this sequence to  construct  C*-algebras $(A_n)_n$ and injective *-homomorphisms
	$\phi_n\colon A_n\to A_{n+1}$ for all $n$.  Then,  by letting  $(k_n)_{n=1}^\infty$ grow sufficiently fast, we will argue that 
	$A_n$ fails to satisfy $\Crm_m$ for all $n$, whence $A$ also fails $\Crm_m$ by Lemma \ref{lem:inductive}.
	
	For each $n\in \NN$, let $X_n=(S^2)^{k_n}$.
	Let $q_n\in M_{2^{k_n}}(C(X_n))$ be the rank one projection on $X_n$ obtained in Example \ref{botttensor}. 
	
	Let $Y_n=\prod_{i=1}^n X_i$. Define a  projection $r_n\in M_\infty(C(Y_n))$ by
	\[
	r_n(y)= e_{11}\oplus q_1(x_1)^{\oplus l_1}\oplus q_2(x_2)^{\oplus l_2}\oplus \cdots \oplus q_n(x_n)^{\oplus l_n},
	\]
	where $y=(x_1,x_2,\dots,x_n)\in Y_n$.   The numbers $l_n$ are defined recursively  such that $l_1=1$ and $l_{n+1}=\mathrm{rank}(r_{n})$ for $n\geq 1$. In fact, it is easy to calculate that $l_n=2^{n-1}$ for $n\geq 1$.
	
	Define
	\[
	A_n = \her(r_n)=r_nM_\infty(C(Y_n))r_n.
	\]
	
	For each $n$, choose  $z_n\in Y_n$. Observe, from the definition of $r_n$, that
	\[
	r_{n+1} = (r_{n}\circ\pi_{Y_n})\oplus (q_{n+1}^{\oplus l_{n+1}})\circ \pi_{X_{n+1}}.
	\]
	Keeping in mind the construction of an injective *-homomorphism discussed before the theorem, let $\phi_n\colon \her(r_n)\to \her(r_{n+1})$ be the *-homomorphism defined as follows:
	\[
	\phi_n(f)=\begin{pmatrix}f\circ\pi_{Y_n} & 0\\ 0 & \mathrm{ev}_{z_n}(f)\otimes (q_{n+1}\circ \pi_{X_{n+1}})\end{pmatrix}.
	\]
	
	Let $A:=\varinjlim (A_n,\phi_n)$ be the inductive limit in the category of C*-algebras.
	In order for $A$ to fail  $\Crm_m$ it suffices that  $A_n$ fails  $\Crm_m$ for all $n$ (Lemma \ref{lem:inductive}).
	Notice that
	\[
	r_n=e_{11}\oplus r_n',
	\]
	where
	\[
	r_n'=(q_1\circ\pi_{X_1})^{\oplus l_1}\oplus (q_2\circ\pi_{X_2})^{\oplus l_2}\oplus \cdots \oplus (q_n\circ \pi_{X_n})^{\oplus l_n}.
	\]
	By Example \ref{ex:directsum}, if $8ml_n\leq k_n$ then $A_n$ does not have $\Crm_m$. Thus, it suffices to choose a sequence $(k_n)_n$
	such that  $8m2^{n-1}\leq k_n$ for all $n$.

	It is known that choosing the points $z_n\in Y_n$ suitably, one can arrange for  the inductive limit $A$ to be simple. More concretely, choose 
	\[
	z_n=(z_{n1},z_{n2},\ldots,z_{nn})\in X_1\times X_2\times \cdots \times X_n
	\] 
such that, for each fixed $j$, the sequence 
	\[
	(z_{n1},z_{n2},\ldots,z_{nj})\in X_1\times X_2\times \cdots \times X_j\hbox{ for }n\geq j
	\]
	is dense in $X_1\times X_2\cdots \times X_j$. Then, using the simplicity criterion from \cite[Proposition 2.1]{dnnp}, we deduce that	$A$ is simple (see \cite[Section 3, p. 1093]{villadsen}).
\end{proof}	

Given a unital C*-algebra $A$ and a polynomial $f$ in noncommuting variables with coefficients in $\CC$, let us denote by $f(A)$ the image of $f$ when evaluated in $A$, i.e., the set
$\{f(\bar a):\bar a\in A^n\}$.

\begin{corollary}\label{cor:fimage}
	Let $f$ be a polynomial in noncommuting variables that is an identity for $\CC$. Then there exists a simple unital
	noncommutative C*-algebra $A$ such that $1\not \in f(A)$.  	
\end{corollary}	
\begin{proof}
	Let $x_1,\ldots,x_n$ be the variables of $f$. Then $f$ belongs to the two-sided ideal generated by the commutators 
	$[x_i,x_j]$  for  $i,j=1,\ldots,n$ with $i\neq j$. That is,
	\[
	f=\sum_{k=1}^m g_k[x_{i_k},x_{j_k}]h_k
	\]
	for some polynomials $g_k,h_k$ and some $m\in \NN$. 
	It follows that if  $1\in f(A)$ then $A$ has property $\Crm_{m}$.  By Theorem \ref{thm:simplenoCm},
	there exist simple unital noncommutative C*-algebras  without $\Crm_m$. 
\end{proof}	

By the obvious induction argument, Theorem \ref{t1} implies that
every element in a finite-dimensional simple algebra $A$
that is not a field is  a product of $n$ commutators for any $n\ge 2$. Corollary 3.8 in particular shows that this is not always true in infinite dimensions. More precisely, the following holds.

\begin{corollary}
    There exists a   simple unital algebra that is not a field in which $1$ is not a product of commutators.
\end{corollary}

\end{document}